\documentclass[10pt]{amsproc}

\usepackage{geometry, enumerate, amsmath, amssymb, amsthm, amscd, mathrsfs, hyperref, graphicx, color, enumerate, caption, subcaption, float, verbatim, tikz, calc}

\usetikzlibrary{decorations.markings}
\usetikzlibrary{arrows, shapes}
\usetikzlibrary{arrows.meta, shapes.misc}

\tikzstyle{vertex}=[circle,draw,inner sep=0pt, minimum size=6pt]

\numberwithin{equation}{section}
\newtheorem{Thm}[equation]{Theorem}
\newtheorem{Cor}[equation]{Corollary}

\newtheorem{Lem}[equation]{Lemma}

\newtheorem{Prop}[equation]{Proposition}

\theoremstyle{definition}
\theoremstyle{definition}\newtheorem{Def}[equation]{Definition}
\theoremstyle{definition}
\theoremstyle{definition}\newtheorem{Not}[equation]{Notation}
\theoremstyle{definition}

\newenvironment{myproof}[1][\proofname]{%
  \proof[\scshape Proof #1]%
}{\endproof}

\newcommand{\F}{\mathbb{F}}
\newcommand{\K}{\mathbb{K}}

\renewcommand{\le}{\leqslant}
\renewcommand{\ge}{\geqslant}
\newcommand{\lcm}{\text{lcm}}

\newcommand{\mcB}{\mathcal{B}}
\newcommand{\mcC}{\mathcal{C}}
\newcommand{\mcD}{\mathcal{D}}

\newcommand{\mcF}{\mathcal{F}}
\newcommand{\mcG}{\mathcal{G}}
\newcommand{\mcL}{\mathcal{L}}
\newcommand{\mcN}{\mathcal{N}}
\newcommand{\mcP}{\mathcal{P}}

\newcommand{\mcR}{\mathcal{R}}
\newcommand{\mcS}{\mathcal{S}}

\newcommand{\mcU}{\mathcal{U}}

\newcommand{\mcV}{\mathcal{V}}

\newcommand{\mfn}{\mathfrak{n}}
\newcommand{\mfN}{\mathfrak{N}}

\newcommand{\msC}{\mathscr{C}}

\newcommand{\row}{\text{row}}

\newcommand{\wGU}{\widehat{\mcG_U}}
\newcommand{\GL}{{\rm GL}}
\newcommand{\Gal}{\text{Gal}}
\newcommand{\Diag}{\text{Diag}}

\title[Null ideals of sets of $3 \times 3$ similar matrices]{Null ideals of sets of $3 \times 3$ similar matrices with irreducible characteristic polynomial}

\date{\today}

\author{Eric Swartz}
\address{Department of Mathematics, William \& Mary, P.O. Box 8795, Williamsburg, VA 23187-8795, USA}
\email{easwartz@wm.edu}

\author{Nicholas J. Werner}
\address{Department of Mathematics, Computer and Information Science, State University of New York College at Old Westbury, Old Westbury, NY 11560, USA}
\email{wernern@oldwestbury.edu}

\begin{document}

\begin{abstract}
Let $F$ be a field and $M_n(F)$ the ring of $n \times n$ matrices over $F$. Given a subset $S$ of $M_n(F)$, the null ideal of $S$ is the set of all polynomials $f$ with coefficients from $M_n(F)$ such that $f(A) = 0$ for all $A \in S$. We say that $S$ is core if the null ideal of $S$ is a two-sided ideal of the polynomial ring $M_n(F)[x]$. We study sufficient conditions under which $S$ is core in the case where $S$ consists of $3 \times 3$ matrices, all of which share the same irreducible characteristic polynomial. In particular, we show that if $F$ is finite with $q$ elements and $|S| \ge q^3-q^2+1$, then $S$ is core. As a byproduct of our work, we obtain some results on block Vandermonde matrices, invertible matrix commutators, and graphs defined via an invertible difference relation.\\

\noindent Keywords: Null ideal, Vandermonde matrix, matrix commutator\\
MSC: 16S50, 15A15, 15B33, 05C25
\end{abstract}

\maketitle 

\thispagestyle{empty}

\section{Introduction}\label{Intro}

The purpose of this paper is to investigate null ideals of sets of matrices. Let $M_n(F)$ be the ring of $n \times n$ matrices with entries from a field $F$, and let $\GL(n,F)$ be the group of invertible matrices in $M_n(F)$. Associate $F$ with the scalar matrices in $M_n(F)$, so that $F \subseteq M_n(F)$. For a prime power $q$, $\F_q$ denotes the field with $q$ elements and $\GL(n,q)$ is the group of invertible matrices in $M_n(\F_q)$. Unless otherwise specified, all vectors are row vectors, and $F^n = \{ \begin{bmatrix} x_1 \; \cdots \; x_n \end{bmatrix} : x_i \in F\}$ is the $F$-vector space of $n$-dimensional row vectors.

Given a subset $S \subseteq M_n(F)$, the \textit{null ideal} of $S$ is $N(S):= \{f \in M_n(F)[x] : f(A)=0 \text{ for all } A \in S\}$. Note that the polynomials in $N(S)$ have matrix coefficients. Since these coefficients come from a noncommutative ring, care must be taken when evaluating polynomials. We will adopt the convention that the indeterminate $x$ is central, and that polynomials in $M_n(F)[x]$ satisfy right evaluation, i.e., that before a polynomial can be evaluated, powers of $x$ must appear to the right of any coefficients. For example, a product $(Ax)(Bx)$ of monomials $Ax$ and $Bx$ is (as usual) equal to $ABx^2$, but it must be expressed in the latter form before it is evaluated at an element of $M_n(F)$. In particular, this means that evaluation at a matrix $C$ defines a multiplicative function $M_n(F)[x] \to M_n(F)$ if and only if $C$ is a scalar matrix. Further details on working with polynomials over noncommutative rings can be found in ring theory texts such as \cite[\S 16]{Lam}.

Most previous papers \cite{Brown1, Brown2, Brown3, HeuRiss, Rissner} on null ideals of matrices have considered only polynomials with scalar coefficients that vanish on $S$. The more recent article \cite{WernerNull} studied null ideals in $M_n(F)[x]$, with a particular focus on sets of $2 \times 2$ matrices. This paper may be considered a sequel to \cite{WernerNull}, in which we extend some results from the $n=2$ case to the $n=3$ case. As in \cite{WernerNull}, the major question that we will investigate is: given $S \subseteq M_n(F)$, is the null ideal $N(S)$ a two-sided ideal of $M_n(F)[x]$? Because $M_n(F)[x]$ contains polynomials with noncommuting coefficients, this may or may not be the case, and it is nontrivial to decide whether or not $N(S)$ is a two-sided ideal.

\begin{Def}
We say that $S \subseteq M_n(F)$ is \textit{core} if $N(S)$ is a two-sided ideal of $M_n(F)[x]$.
\end{Def}

A major motivation for studying null ideals is their connection to \textit{integer-valued polynomials}.  This topic has a rich history over commutative rings (see, for instance, \cite{CC}), and researchers have begun studying noncommutative versions of this topic with a focus on integer-valued polynomials on matrix rings in recent years (see \cite{WernerSurvey} for a recent survey of results).  In particular, if $D$ is a commutative integral domain with field of fractions $K$ and $S \subseteq M_n(D)$, and
\[ {\rm Int}(S, M_n(D)) := \{ f \in M_n(K)[x] : f(A) \in M_n(D) \text{ for all } A \in S\},\]
a major open question in this area is whether or not ${\rm Int}(S, M_n(D))$ is even a ring.  As is noted in \cite[Section 2]{WernerNull} (to which the interested reader is directed for relevant details), one way to approach this question is to translate it into a question about null ideals and core sets.

Thus, we seek to characterize core subsets of $M_n(F)$. In general, this is quite difficult to do. In \cite{WernerNull}, a strategy was outlined to determine when a finite subset $S$ of $M_2(F)$ is core. This process was based off of the observation (\cite[Prop.\ 3.1(3)]{WernerNull}) that any full $\GL(n,F)$-conjugacy class in $M_n(F)$ is core. For a general subset $S \subseteq M_2(F)$, we first partition $S$ into similarity classes and then examine each class individually to see whether or not it is core. Then, we consider what happens when these classes are united back into $S$. With this approach, similarity classes of matrices with an irreducible characteristic polynomial hold a special status.

\begin{Def}
For any $A \in M_n(F)$, let $\mu_A$ be the minimal polynomial of $A$. Given $f \in F[x]$, we define $\mcC_n(f) := \{A \in M_n(F) : \mu_A=f\}$. Note that when $n=2$ or $n=3$, for any $A \in \mcC_n(f)$ the similarity class of $A$ in $M_n(F)$ is equal to $\mcC_n(f)$.
\end{Def}

\begin{Prop}\label{prop:2x2 case}
Let $S \subseteq M_n(F)$ be finite and nonempty. Let $\mcP = \{\mu_A : A \in S\}$. Then, $S = \bigcup_{f \in \mcP} (S \cap \mcC_n(f))$.
\begin{enumerate}[(1)]
\item \cite[Prop.\ 3.1(4)]{WernerNull} If each class $S \cap \mcC_n(f)$ is core, then $S$ is core.
\item \cite[Cor.\ 5.4]{WernerNull} If $f \in \mcP$ is irreducible and $S \cap \mcC_n(f)$ is not core, then $S$ is not core.
\item \cite[Cor.\ 5.9]{WernerNull} Let $f \in F[x]$ be an irreducible quadratic and let $T \subseteq \mcC_2(f)$ be nonempty. Then, $T$ is core if and only if $|T| \ge 2$.
\end{enumerate}
\end{Prop}

Unfortunately, Proposition \ref{prop:2x2 case}(2) need not hold if the polynomial $f$ is reducible (see \cite[Ex.\ 5.5]{WernerNull}, where an example is given of a core set $S \subseteq M_2(F)$ for which each subset $S \cap \mcC_2(f)$ is non-core). Nevertheless, for irreducible $f$, this result highlights the importance of subsets of $\mcC_n(f)$ in the determination of core sets. 

The purpose of this paper is to study core subsets of $\mcC_3(f)$, where $f$ is irreducible of degree 3. For the remainder of the article, we will assume that $F$ is such that $F[x]$ contains an irreducible cubic polynomial $m(x):=x^3+ax^2+bx+c$. We will pay particular attention to the case where $F$ is a finite field. Since all matrices under consideration will be $3 \times 3$, we will let $\mcC(m):=\mcC_3(m)$. 

In the $3 \times 3$ case, we cannot obtain a characterization of core subsets of $\mcC(m)$ that is as clean as Proposition \ref{prop:2x2 case}(3). Moreover, the theorems we can prove require significantly more effort than in the $2 \times 2$ case. It is easy to see that for $S \subseteq \mcC(m)$ to be core, it is necessary that $|S| \ge 3$ (see Lemma \ref{lem:small sets} below). Our main results (Corollaries \ref{cor:singular case}, \ref{cor:mixed case core}, \ref{cor:DA core bound}) provide several conditions under which $S$ is guaranteed to be core. When these are applied in the case where $F$ is finite, we obtain a simple sufficient condition on $|S|$ for $S$ to be core.

\begin{Thm}\label{thm:main}
Assume that $F=\F_q$, and let $S \subseteq \mcC(m)$. If $|S| \ge q^3-q^2+1$, then $S$ is core.
\end{Thm}

When $F=\F_q$ and $3 \le |S| \le q^3-q^2$, $S$ may or may not be core. In fact, for each $3 \le s \le q^3-q^2$, there exists a subset $S \subseteq \mcC(m)$ with $|S| = s$ that is core (see Corollary \ref{cor:mixed case core}) and a subset $T \subseteq \mcC(m)$ with $|T| = s$ that is not core (see Proposition \ref{prop:max non-core subset}).  In particular, these results show that the problem is far less tractable in the $3 \times 3$ case than the $2 \times 2$ case, and make the authors pessimistic that general results can easily be obtained in the $n \times n$ case for $n > 2$. 

Our techniques largely rely on calculations with block Vandermonde matrices (discussed in Section \ref{noncore section}) whose blocks are elements of $\mcC(m)$. In Theorem \ref{thm:singular case}, Theorem \ref{thm:mixed case}, Lemma \ref{lem:inv vandy lemma}, and Corollary \ref{cor:vandy invertible case}, we give sufficient conditions for these block Vandermonde matrices to be invertible. To prove these results, we are often led to the question of whether or not the difference of two elements of $\mcC(m)$ is invertible. Problems of this sort are closely related to that of characterizing invertible commutators in matrix rings, of which there is an extensive body of literature (see e.g.\ \cite{GurLan, KhuranaLam, Marq, MartinsSilva1, MartinsSilva2, Uhlig}).  Hence, while our motivation is to identify core subsets of $\mcC(m)$, our results may also be of interest to researchers working with block Vandermonde matrices or invertible commutators.

In Section \ref{noncore section}, we provide examples of core and non-core subsets of $M_3(F)$, and discuss how block Vandermonde matrices can help to decide whether or not a subset $S \subseteq \mcC(m)$ is core. This approach splits naturally into three cases, which we examine in Sections \ref{singular section}, \ref{mixed section}, and \ref{invertible section}. The proof of Theorem \ref{thm:main} is given at the conclusion of Section \ref{invertible section}.  Finally, we close the paper with a short section (Section \ref{graph section}) highlighting some interesting combinatorial consequences of our work.  Graphs related to matrices have received attention in recent years \cite{HHLS, SM}.  Moreover, sets of matrices whose pairwise differences are invertible have been of interest to researchers in finite geometry and coding theory (see, for example, \cite{KantorI, KantorII, Sheekey}). Thus, the techniques and theorems we use to characterize core subsets of $\mcC(m)$ have connections to broader areas of discrete and combinatorial mathematics.

\section{Examples of non-core sets}\label{noncore section}

As discussed in the introduction, we are interested in finding core sets $S \subseteq M_3(F)$, which are those sets for which the null ideal $N(S)$ is a two-sided ideal of $M_3(F)[x]$. For any $S \subseteq M_3(F)$, $N(S)$ is a left ideal of $M_3(F)[x]$ by \cite[Prop.\ 3.1(1)]{WernerNull}. A singleton set $\{A\}$ is core if and only if $A$ is scalar \cite[Lem.\ 3.3]{WernerNull}, while a full similarity class in $M_3(F)$ is always core \cite[Prop.\ 3.1]{WernerNull}. In general, by \cite[Prop.\ 4.1]{WernerNull}, $S$ is core if and only if $N(S)$ is generated by a polynomial with scalar coefficients. Specifically, for any $A \in M_3(F)$, let $\mu_A$ be the minimal polynomial of $A$. Define $\phi_S(x) := \lcm\{\mu_A(x)\}_{A \in S}$ to be the (monic) least common multiple in $F[x]$ of all the minimal polynomials of the elements of $S$, if such a polynomial exists. If not, we take $\phi_S(x)=0$. Clearly, $\phi_S \in N(S)$.

\begin{Prop}\label{prop:phi_S}
Let $S \subseteq M_3(F)$ and assume that $\phi_S \ne 0$.
\begin{enumerate}[(1)]
\item \cite[Thm.\ 4.4]{WernerNull} $S$ is core if and only if $N(S)$ is generated (as a two-sided ideal of $M_3(F)[x]$) by $\phi_S$.
\item \cite[Cor.\ 4.6]{WernerNull} $S$ is not core if and only if $N(S)$ contains a polynomial of degree less than $\deg \phi_S$.
\end{enumerate}
\end{Prop}

Because of Proposition \ref{prop:phi_S}, to determine if $S$ is core it suffices to consider the polynomials in $N(S)$, rather than the ideal structure of $M_3(F)[x]$. Clearly, this task will be easier when $\deg \phi_S$ is small. Hence, we shall focus on subsets $S \subseteq \mcC(m)$, where $m(x):=x^3+ax^2+bx+c \in F[x]$ is irreducible. Given $A \in \mcC(m)$, we have $\mu_A = m$, so $\phi_S=m$ for all (nonempty) $S \subseteq \mcC(m)$. Since $m$ is irreducible, $F[A]$ (the $F$-subalgebra of $M_3(F)$ generated by $A$) is isomorphic to the field $F[x]/(m)$. In particular, when $F=\F_q$, we have $\F_q[A] \cong \F_{q^3}$, and the centralizer of $A$ in $\GL(3,q)$ has order $q^3-1$. Consequently, in this case $\mcC(m)$---being the full conjugacy class of $A$---has cardinality 
\begin{equation}
\label{eq:Cmsize}
|\mcC(m)| = |\GL(3,q)|/(q^3-1) = (q^3-q)(q^3-q^2).
\end{equation}

Most of our work is based off of calculations with Vandermonde matrices. Given matrices $A_1, \ldots, A_d$ in $M_3(F)$, we define $\mcV(A_1, \ldots, A_d)$ to be the following block Vandermonde matrix:
\begin{equation*}
\mcV(A_1, \ldots, A_d) := \begin{bmatrix}
I & I & \cdots & I\\ A_1 & A_2 & \cdots & A_d\\ A_1^2 & A_2^2 & \cdots & A_d^2 \\ \vdots & \vdots & \ddots & \vdots \\ A_1^{d-1} & A_2^{d-1} & \cdots & A_d^{d-1}
\end{bmatrix}
\end{equation*}
These block matrices are closely tied to polynomial evaluation. Given a degree $d-1$ polynomial $f(x) = \sum_{i=0}^{d-1} \alpha_i x^i$ with coefficients from $M_3(F)$, we have
\begin{equation}\label{Vandy}
\begin{bmatrix}\alpha_0 & \alpha_1 & \cdots & \alpha_{d-1} \end{bmatrix} \mcV(A_1, \ldots, A_d) = \begin{bmatrix}f(A_1) & f(A_2) & \cdots & f(A_d) \end{bmatrix}
\end{equation}
Using the relation \eqref{Vandy}, we can test whether a polynomial is in $N(S)$. We can also give a sufficient condition for $S$ to be core.

\begin{Lem}\label{lem:Vandy}\cite[Prop.\ 4.9]{WernerNull} Let $S \subseteq M_3(F)$ and let $d = \deg \phi_S$. If there exist $A_1, \ldots, A_d \in S$ such that $\mcV(A_1, \ldots, A_d)$ is invertible, then $S$ is core. In particular, if $S \subseteq \mcC(m)$ and there exist $A, B, C \in S$ such that $\mcV(A,B,C)$ is invertible, then $S$ is core.
\end{Lem}
\begin{proof}
This is shown in \cite{WernerNull}, but we include the short proof for the sake of completeness. Suppose that $A_1, \ldots, A_d \in S$ and $\mcV:=\mcV(A_1, \ldots, A_d)$ is invertible. If $f(x) = \sum_{i=0}^{d-1} \alpha_i x^i \in N(S)$, then the right-hand side of \eqref{Vandy} is $[0 \; \cdots \; 0]$. Since $\mcV$ is invertible, each $\alpha_i=0$. Thus, $f=0$, $N(S)$ contains no polynomial of degree less than $d$, and $S$ is core by Proposition \ref{prop:phi_S}(2).
\end{proof}

\begin{Lem}\label{lem:small sets}\mbox{}
\begin{enumerate}[(1)]
\item Let $A, B \in M_3(F)$. Then, there exists a polynomial of degree at most 2 in $N(\{A, B\})$.
\item Let $S \subseteq \mcC(m)$ be nonempty. If $|S| \le 2$, then $S$ is not core.
\end{enumerate}
\end{Lem}
\begin{proof}
(1) Let $\mcV$ be the Vandermonde matrix
\begin{equation*}
\mcV = \begin{bmatrix} I & I \\ A & B \\ A^2 & B^2 \end{bmatrix}.
\end{equation*}
We seek matrices $\alpha_0$, $\alpha_1$, and $\alpha_2$ such that $\begin{bmatrix} \alpha_0 & \alpha_1 & \alpha_2\end{bmatrix}\mcV=0$. If such an equation holds, then by \eqref{Vandy} the polynomial $\alpha_2 x^2 + \alpha_1 x + \alpha_0$ is in $N(\{A, B\})$. 

To find the matrices $\alpha_0$, $\alpha_1$, and $\alpha_2$, we can perform block row operations on $\mcV$. If $B-A$ is invertible, then let $U= (B-A)^{-1}$ and take
\begin{equation*}
\begin{bmatrix}\alpha_0 & \alpha_1 & \alpha_2\end{bmatrix} = \begin{bmatrix}0 & 0 & I\end{bmatrix} \begin{bmatrix} I & 0 & 0 \\ 0 & I & 0\\ 0 & -(B^2-A^2)U & I\end{bmatrix} \begin{bmatrix} I & 0 & 0 \\ -A & I & 0 \\ -A^2 & 0 & I \end{bmatrix}.
\end{equation*}
Then $\alpha_2=I$ and $\begin{bmatrix}\alpha_0 & \alpha_1 & \alpha_2\end{bmatrix}\mcV=0$, so $x^2 + \alpha_1 x + \alpha_0$ is a monic quadratic polynomial in $N(\{A,B\})$. If $B-A$ is not invertible, then let $\alpha_1$ be nonzero and such that $\alpha_1(B-A) = 0$. Then, $\alpha_1 x - \alpha_1 A \in N(\{A,B\})$.

(2) By assumption, $m$ is an irreducible cubic, so $\phi_S = m$. If $|S|=1$, then $S$ is not core by \cite[Lem.\ 3.3(1)]{WernerNull}. If $|S|=2$, then by (1) $N(S)$ contains a polynomial of degree at most 2. Hence, $S$ is not core by Proposition \ref{prop:phi_S}(2).
\end{proof}

Next, we will show that $\mcC(m)$ contains non-core subsets of cardinality greater than 2. In particular, when $F=\F_q$, $\mcC(m)$ contains non-core subsets with $s$ elements for each $1 \le s \le q^3-q^2$. Recall that $m(x)=x^3+ax^2+bx+c$ is irreducible, so $c \ne 0$.

\begin{Lem}\label{lem:counting mats}
Let $B \in \mcC(m)$ have the form 
\begin{equation*}
B = \begin{bmatrix}0&0&-c\\x_1&x_2&x_3\\x_4&x_5&x_6\end{bmatrix}
\end{equation*}
for some $x_1, \ldots, x_6 \in F$. Then, $x_5 \ne 0$, and $B$ is completely determined by $x_2$, $x_4$, and $x_5$. Thus, when $F=\F_q$, the set $\mcC(m)$ contains $q^3-q^2$ matrices with first row equal to $\begin{bmatrix}0&0&-c\end{bmatrix}$.
\end{Lem}
\begin{proof} First, note that such a $B$ must exist, because the companion matrix $\left[\begin{smallmatrix}0&0&-c\\1&0&-b\\0&1&-a\end{smallmatrix}\right]$ of $m$ is an element of $\mcC(m)$. Next, if $x_5=0$, then $\begin{bmatrix}0&1&0\end{bmatrix}^T$ is a right eigenvector of $B$ with eigenvalue $x_2$. This contradicts the fact that $m$ is irreducible. So, $x_5 \ne 0$. 

Computing the characteristic polynomial of $B$ yields
\begin{equation}\label{char poly}
m(x) = x^3 + (-x_2-x_6)x^2 + (cx_4-x_3x_5+x_2x_6)x + c(x_1x_5-x_2x_4).
\end{equation}
This forces $x_1x_5-x_2x_4=1$, and since $x_5 \ne 0$ we obtain
\begin{equation}\label{x1 eq}
x_1=x_5^{-1}(1+x_2x_4).
\end{equation}
Now, from \eqref{char poly} we have $b = cx_4-x_3x_5+x_2x_6$. Solving this equation for $x_3$ gives $x_3 = x_5^{-1}(cx_4+x_2x_6-b)$. But, by \eqref{char poly}, $x_6 = -x_2-a$. Using this substitution produces
\begin{equation}\label{x3 eq}
x_3 = x_5^{-1}(cx_4+x_2(-x_2-a)-b) = x_5^{-1}(-x_2^2-ax_2+cx_4-b).
\end{equation}
Equations \eqref{x1 eq} and \eqref{x3 eq}, along with the relation $x_6 = -x_2-a$, show that all of the entries of $B$ are determined by $x_2$, $x_4$, and $x_5$. When $F=\F_q$, there are $q$ choices for each of $x_2$ and $x_4$, and $q-1$ choices for $x_5$. Thus, in that case, the number of matrices $B$ in $\mcC(m)$ is equal to $q^2(q-1)=q^3-q^2$.
\end{proof}

\begin{Prop}\label{prop:max non-core subset}
Let $T$ be the set of all matrices in $\mcC(m)$ with first row equal to $\begin{bmatrix}0&0&-c\end{bmatrix}$. Then, every nonempty subset of $T$ is not core. When $F=\F_q$, we have $|T|=q^3-q^2$, and for each $1\le s \le q^3-q^2$, $\mcC(m)$ contains a non-core subset with $s$ elements.
\end{Prop}
\begin{proof} 
Let $T' \subseteq T$ be nonempty. Since every matrix in $T'$ has first row $\begin{bmatrix}0&0&-c\end{bmatrix}$, the polynomial 
\begin{equation*}
f(x) = \left[\begin{smallmatrix}1&0&0\\0&0&0\\0&0&0\end{smallmatrix}\right]x - \left[\begin{smallmatrix}0&0&-c\\0&0&0\\0&0&0\end{smallmatrix}\right]
\end{equation*}
is in $N(T')$. But, $\phi_{T'}=m$, which has degree 3, so $T'$ is not core by Proposition \ref{prop:phi_S}(2). If $F$ is finite of order $q$, then $|T|=q^3-q^2$ by Lemma \ref{lem:counting mats}. Hence, $1 \le |T'| \le q^3-q^2$.
\end{proof}

Much of the remainder of this paper is devoted to showing that, when $F=\F_q$, the subset $T$ in Proposition \ref{prop:max non-core subset} has the largest cardinality of any non-core subset of $\mcC(m)$. In other words, if $F=\F_q$, $S \subseteq \mcC(m)$, and $|S| \ge q^3-q^2+1$, then $S$ is core. 

We end this section with a generalization of Lemma \ref{lem:counting mats} that will be used later. Let $e_1 = \begin{bmatrix}1&0&0\end{bmatrix}$. Then, two matrices $A$ and $B$ share the same first row if and only if $e_1 A = e_1 B$. Given $A \in \mcC(m)$, rather than focusing on matrices that share the same first row as $A$, we can consider matrices $B$ such that $vA=vB$ for some nonzero $v \in F^3$.

\begin{Def}\label{def:E_A,v}
For each $A \in \mcC(m)$ and each nonzero $v \in F^3$, define $E_{A,v} := \{B \in \mcC(m) : vB=vA\}$. 
\end{Def}

\begin{Lem}\label{lem:E_A,v}
Let $A \in \mcC(m)$. Let $v \in F^3$ be nonzero.
\begin{enumerate}[(1)]
\item If $U$ is invertible, then $U^{-1} (E_{A,v}) U = E_{U^{-1}AU,vU}$. 
\item $|E_{A,v}| = |E_{A,e_1}|$.
\item If $F=\F_q$, then $|E_{A,v}| = q^3-q^2$.
\end{enumerate}
\end{Lem}
\begin{proof}
(1) Let $U$ be invertible. Given $B \in E_{A,v}$, we have
\begin{equation*}
(vU)(U^{-1}BU) = v(BU) = v(AU) = (vU)(U^{-1}AU). 
\end{equation*}
Thus, $U^{-1} (E_{A,v}) U \subseteq E_{U^{-1}AU,vU}$. Conversely, if $B \in E_{U^{-1}AU,vU}$, then 
\begin{equation*}
v(UBU^{-1}) = (vU)B(U^{-1}) = (vU)(U^{-1}AU)(U^{-1}) = vA
\end{equation*}
so $UBU^{-1} \in E_{A,v}$.

(2) Let $\mcF=F[A]$. Since $m$ is irreducible, $\mcF$ is isomorphic to the field $F[x]/(m)$ and $\mcF$ acts transitively on $F^3$. 
So, there exists an invertible matrix $U \in \mcF$ such that $vU=e_1$; note that $U$ commutes with $A$. By part (1), $U^{-1}(E_{A,v})U = E_{A,e_1}$. Since conjugation by $U^{-1}$ is a bijection, we get $|E_{A,v}| = |U^{-1}(E_{A,v})U|=|E_{A,e_1}|$.

(3) Assume $F=\F_q$ and let $C$ be the companion matrix of $m$. Then, $C \in \mcC(m)$, and there exists an invertible matrix $P$ such that $C = P^{-1}AP$. Let $w=e_1P^{-1}$. Then, by part (1), $E_{C,e_1} = P^{-1}(E_{A,w})P$, and so $|E_{C,e_1}| = |E_{A,w}|=|E_{A,e_1}|$ by part (2). Since $C$ has first row $\begin{bmatrix}0&0&-c\end{bmatrix}$, $|E_{C,e_1}|=q^3-q^2$ by Lemma \ref{lem:counting mats}, which completes the proof.
\end{proof}

\section{Singular Differences}\label{singular section}

As in Section \ref{noncore section}, assume throughout that $m(x)=x^3+ax^2+bx+c \in F[x]$ is irreducible. For any nonempty subset $S \subseteq \mcC(m)$, we know that $\phi_S=m$, which has degree three. Hence, by Lemma \ref{lem:Vandy}, we can conclude that $S$ is core if we can find three matrices $A, B, C \in S$ such that the Vandermonde matrix $\mcV:=\mcV(A,B,C)$ is invertible. Given such matrices $A$, $B$, and $C$, we can perform block row operations on $\mcV(A,B,C)$ to produce the matrix
\begin{equation*}
\mcV_1 := \begin{bmatrix} I&I&I\\0&B-A&C-A\\0&B^2-A^2&C^2-A^2\end{bmatrix}.
\end{equation*}
Let $\mcV_2:=\begin{bmatrix}B-A&C-A\\B^2-A^2&C^2-A^2\end{bmatrix}$. 
Then, $\det \mcV =\det \mcV_1  = \det \mcV_2$, and sufficient conditions under which $S$ is core can be found by studying the matrix $\mcV_2$.

We will consider three cases, depending on whether both $B-A$ and $C-A$ are singular; one of $B-A$ or $C-A$ is singular and the other is invertible; or both $B-A$ and $C-A$ are invertible. In this section, we will focus on the first case, where both differences are singular.

\begin{Def}\label{def:nullspace}
For any matrix $A \in M_3(F)$, let $\mfn(A)$ be the left nullspace of $A$ and let $\mfn_r(A)$ be the right nullspace of $A$. So, elements of $\mfn(A)$ are row vectors, and elements of $\mfn_r(A)$ are column vectors. Given two matrices $A, B \in \mcC(m)$, define $\mfN_{AB}:=\mfn(B-A) \oplus \mfn(B^2-A^2)$.
\end{Def}

\begin{Lem}\label{lem:1dimnull}
Let $A, B \in \mcC(m)$ such that $A \ne B$. If $B-A$ is singular, then $\mfn(B-A)$ is 1-dimensional. Moreover, for all $\lambda \in F$, if $B^2-A^2 - \lambda(B-A)$ is singular, then its nullspace is 1-dimensional.
\end{Lem}
\begin{proof}
Assume that $B-A$ is not invertible. Since $B$ and $A$ are similar and $m$ is irreducible, \cite[Thm.\ 2.2]{Marq} indicates that the rank of $B-A$ does not equal 1. Since $B-A$ is a nonzero $3 \times 3$ singular matrix, its rank must be 2. Hence, its nullspace has dimension one. 

Next, note that $F[A]/F$ is a degree 3 field extension of $F$. For each $\lambda \in F$, $A^2-\lambda A$ is not a scalar matrix (because if so, then $A$ would satisfy a quadratic polynomial from $F[x]$), so $A^2-\lambda A$ generates $F[A]$ over $F$. Thus, the minimal polynomial of $A^2-\lambda A$ is an irreducible cubic. Since $B^2-\lambda B$ is similar to $A^2 -\lambda A$, another appeal to \cite[Thm.\ 2.2]{Marq} shows that when $(B^2-\lambda B) - (A^2 - \lambda A) = (B^2-A^2) - \lambda(B-A)$ is not invertible, it has a 1-dimensional nullspace.
\end{proof}

\begin{Prop}\label{prop:nullspaceint}
Let $A, B \in \mcC(m)$, where $m(x) = x^3 + ax^2 +bx +c$, such that $A \ne B$ and $B-A$ is singular.
\begin{enumerate}[(1)]
\item For all $\lambda \in F$, $\mfn(B-A)\cdot(A+a+\lambda) = \mfn(B^2-A^2-\lambda(B-A))$. In particular, $B^2-A^2-\lambda(B-A)$ is singular.
\item For all $\lambda \in F$, $(B+a+\lambda)\cdot \mfn_r(B-A) = \mfn_r(B^2-A^2-\lambda(B-A))$.
\item $\mfn(B-A) \cap \mfn(B^2-A^2)=\{0\}$ and $\mfn_r(B-A) \cap \mfn_r(B^2-A^2)=\{0\}$.
\item $\dim \mfN_{AB} = 2$, and for any nonzero $v \in \mfn(B-A)$, $\{v, v(A+a)\}$ is a basis for $\mfN_{AB}$.
\end{enumerate}
\end{Prop}
\begin{proof} 

For (1), note that for any matrices $A$ and $B$, we have $B^3-A^3 = A(B^2-A^2) + (B-A)B^2$. Since $A, B \in \mcC(m)$, we know that $m(A)=m(B)=0$. So, $B^3-A^3 = -a(B^2-A^2)-b(B-A)$, and
\begin{equation}\label{power relation}
-a(B^2-A^2)-b(B-A) = A(B^2-A^2) + (B-A)B^2.
\end{equation}
Given $v \in \mfn(B-A)$ and multiplying both sides of \eqref{power relation} on the left by $v$ produces 
\begin{equation}\label{power relation 2}
-av(B^2-A^2) = vA(B^2-A^2).
\end{equation}
Now, let $\lambda \in F$ and consider $v(A+a+\lambda)(B^2-A^2-\lambda(B-A))$. Note that $v(A+a+\lambda)$ is nonzero, because $A+a+\lambda \in F[A]$ is invertible. Using \eqref{power relation 2} and the fact that $v(B-A) = 0$, we obtain
\begin{align*}
v(A+a+\lambda)&(B^2-A^2-\lambda(B-A)) \\
&= vA(B^2-A^2) - \lambda vA(B-A) + (a+\lambda)v(B^2-A^2) - (a+\lambda)\lambda v(B-A)\\
&= -av(B^2-A^2) - \lambda vA(B-A) + (a+\lambda)v(B^2-A^2)\\
&= -\lambda v(A(B-A) - (B^2-A^2))\\
&= -\lambda v (A-B)B = 0.
\end{align*}
This shows that $B^2-A^2-\lambda(B-A)$ is singular, and furthermore that $\mfn(B-A) \cdot (A+a+\lambda) \subseteq \mfn(B^2-A^2-\lambda(B-A))$.  Equality of the subspaces now follows from the fact that they are both 1-dimensional.

For (2), note that $-a(B^2-A^2)-b(B-A)=B^3-A^3 = (B^2-A^2)B + A^2(B-A)$. Hence,
\begin{equation*}
-a(B^2-A^2)w = (B^2-A^2)Bw
\end{equation*}
for any $w \in \mfn_r(B-A)$. The proof now proceeds as in part (1).

For (3), let $\mcN = \mfn(B-A) \cap \mfn(B^2-A^2)$, and suppose that $\mcN \ne \{0\}$. By Lemma \ref{lem:1dimnull}, both $\mfn(B-A)$ and $\mfn(B^2-A^2)$ are 1-dimensional. Hence, $\mcN = \mfn(B-A)=\mfn(B^2-A^2)$. But then, by part (1), $\mcN$ is an invariant subspace of $A+a$, which contradicts the fact that $A+a$ has an irreducible characteristic polynomial. The same argument shows that $\mfn_r(B-A) \cap \mfn_r(B^2-A^2) = \{0\}$. 

Finally, (4) follows from Lemma \ref{lem:1dimnull}, part (3), and part (1).
\end{proof}

Part (1) of Proposition \ref{prop:nullspaceint} shows that if $B-A$ is not invertible, then $B^2-A^2-\lambda(B-A)$ is not invertible for any $\lambda \in F$. In fact, the following strong converse holds.

\begin{Cor}\label{cor:inv diff}
For all $A, B \in \mcC(m)$, the following are equivalent.
\begin{enumerate}[(i)]
\item $B-A$ is singular.
\item For all $\lambda \in F$, $B^2-A^2 - \lambda(B-A)$ is singular.
\item There exists $\lambda \in F$ such that $B^2-A^2 - \lambda(B-A)$ is singular.
\end{enumerate}
\end{Cor}
\begin{proof}
(i) $\Rightarrow$ (ii) is Proposition \ref{prop:nullspaceint}(1), and (ii) $\Rightarrow$ (iii) is obvious. So, it suffices to prove (iii) $\Rightarrow$ (i). Let $\lambda \in F$ such that $B^2-A^2 - \lambda(B-A)$ is not invertible. Then, there exists a nonzero vector $w \in F^3$ such that $w(B^2-A^2 - \lambda(B-A))=0$. 

Let $U = A^2-\lambda A + \lambda(a+\lambda) + b$. Then, $U \ne 0$, because if so, then $A$ would satisfy the quadratic polynomial $x^2 - \lambda x + \lambda(a+\lambda) + b$. However, this is impossible because $m$ is irreducible. So, $U \ne 0$, and since $U$ is an element of the field $F[A]$, it is invertible. Let $v = wU$, which is a nonzero element of $F^3$. We will show that $v(B-A) = 0$.

Certainly, $v(B-A) = wUB - wUA$. By assumption, $w(B^2-A^2 - \lambda (B-A))=0$, which means that $w(A^2-\lambda A)=w(B^2-\lambda B)$. So,
\begin{align*}
wUB &= w(A^2-\lambda A + \lambda(a+\lambda) + b)B\\ &= w(B^2-\lambda B + \lambda(a+\lambda) + b)B\\
&= w(B^3-\lambda B^2 + (\lambda(a+\lambda) + b)B).
\end{align*}
Since $m(B)=0$, we have $B^3=-aB^2-bB-c$. Substituting this into the last expression and simplifying, we obtain
\begin{equation}\label{w eq 1}
wUB = w((-a-\lambda)(B^2-\lambda B) - c).
\end{equation}
The same steps will show that
\begin{equation}\label{w eq 2}
wUA = w((-a-\lambda)(A^2-\lambda A) - c).
\end{equation}
From \eqref{w eq 1}, \eqref{w eq 2}, and the equation $w(A^2-\lambda A)=w(B^2-\lambda B)$, we see that $wUB=wUA$. Hence, $v(B-A)=0$ and $B-A$ is not invertible.
\end{proof}

Next, we show that the subspace $\mfN_{AB}$ defined prior to Lemma \ref{lem:1dimnull} is closely related to the row spaces of $B-A$ and $B^2-A^2$. For $X \in M_3(F)$, let $\row(X)$ denote the row space of $X$.

\begin{Lem}\label{lem:Rspacelem}
Let $A, B \in \mcC(m)$ such that $A \ne B$ and $B-A$ is singular. Let $\mcR$ be the intersection of $\row(B-A)$ and $\row(B^2-A^2)$. 
\begin{enumerate}[(1)]
\item $\dim \mcR = 1$.
\item $\mfN_{AB}\cdot(B-A) = \mcR= \mfN_{AB}\cdot(B^2-A^2)$.
\item Let $v \in F^3$. If $v(B-A) \in \row(B^2-A^2)$, then $v \in \mfN_{AB}$. Likewise, if $v(B^2-A^2) \in \row(B-A)$, then $v \in \mfN_{AB}$.
\end{enumerate}
\end{Lem}
\begin{proof}
(1) By Lemma \ref{lem:1dimnull}, both $B-A$ and $B^2-A^2$ have 1-dimensional nullspaces. Hence, both matrices have 2-dimensional rowspaces. So, $\dim \mcR$ is either 1 or 2. If $\dim \mcR = 2$, then $\row(B-A) = \row(B^2-A^2)$. But then, the right nullspaces (being orthogonal to the row spaces) of $B-A$ and $B^2-A^2$ would be equal. This contradicts Proposition \ref{prop:nullspaceint} (3). So, $\dim \mcR = 1$.

(2) It is clear that $\mfN_{AB} \cdot (B-A) \subseteq \row(B-A)$, so to prove that $\mfN_{AB} \cdot (B-A) = \mcR$, it suffices to show that $\mfN_{AB} \cdot (B-A) \subseteq \row(B^2-A^2)$. We have
\begin{equation*}
\mfN_{AB} \cdot (B-A) = \mfn(B^2-A^2) \cdot (B-A) = \mfn(B-A) \cdot (A+a)(B-A),
\end{equation*}
where the second equality follows from Proposition \ref{prop:nullspaceint}(1). Now, one my check that
\begin{equation*}
(A+a)(B-A) = B^2-A^2 - (B-A)(B-a).
\end{equation*}
Since $\mfn(B-A) \cdot (B-A) = 0$, we get $\mfn(B-A) \cdot (A+a)(B-A) = \mfn(B-A) \cdot (B^2-A^2)$. Thus, $\mfN_{AB} \cdot (B-A) = \mfn(B-A) \cdot (B^2-A^2) \subseteq \row(B^2-A^2)$, and we conclude that $\mfN_{AB} \cdot (B-A) = \mcR$. 

The proof that $\mfN_{AB} \cdot (B^2-A^2) = \mcR$ is similar, and it is enough to show that $\mfN_{AB} \cdot (B^2-A^2) \subseteq \row(B-A)$. Since $\mfN_{AB}\cdot (B^2-A^2) = \mfn(B-A)\cdot (B^2-A^2)$ and $B^2-A^2 = A(B-A)+(B-A)B$, we obtain
\begin{equation*}
\mfN_{AB} \cdot (B^2-A^2) = \mfn(B-A) \cdot (A(B-A)) \subseteq \row(B-A),
\end{equation*}
as desired.

(3) Let $v \in F^3$, let $x=v(B-A)$, and assume that $x \in \row(B^2-A^2)$. Then, $x \in \mcR$, so by (2) $x= w(B-A)$ for some $w \in \mfN_{AB}$. Let $w = w_1 + w_2$, where $w_1 \in \mfn(B-A)$ and $w_2 \in \mfn(B^2-A^2)$. Then, $x=w(B-A)=w_2(B-A)$. Thus, $v-w_2 \in \mfn(B-A)$, which means that $v \in \mfN_{AB}$. The proof of the other claim is similar.
\end{proof}

We can now characterize those matrices $A, B, C \in \mcC(m)$ with pairwise singular differences that produce a singular block Vandermonde matrix.

\begin{Thm}\label{thm:singular case}
Let $A, B, C \in \mcC(m)$ be distinct and such that both $B-A$ and $C-A$ are singular. Then, the following are equivalent.
\begin{enumerate}[(i)]
\item $\mcV(A,B,C)$ is not invertible.
\item $\mcV_2:=\begin{bmatrix} B-A & C-A\\ B^2-A^2 & C^2-A^2 \end{bmatrix}$ is not invertible.
\item $\mfN_{AB}=\mfN_{AC}$.
\item $\mfn(B-A) = \mfn(C-A)$.
\item There exists a nonzero $v \in F^3$ such that $vA=vB=vC$.
\end{enumerate}
\end{Thm}
\begin{proof} (i) $\Leftrightarrow$ (ii) Perform row operations on $\mcV(A,B,C)$ as described at the start of Section \ref{singular section}. 

(iv) $\Leftrightarrow$ (v) By Lemma \ref{lem:1dimnull}, both $\mfn(B-A)$ and $\mfn(C-A)$ are 1-dimensional, so (v) implies (iv). The other implication is clear.

(iv) $\Rightarrow$ (iii) Assume that $\mfn(B-A) = \mfn(C-A)$ and let $v \in \mfn(B-A)$ be nonzero. By Proposition \ref{prop:nullspaceint}(4), $\{v, v(A+a)\}$ is a basis for $\mfN_{AB}$. The same is true for $\mfn(C-A)$ and $\mfN_{AC}$. Thus, $\mfN_{AB} = \mfN_{AC}$.

(iii) $\Rightarrow$ (iv) Assume that $\mfN_{AB} = \mfN_{AC}$. Let $v$ be a basis vector for $\mfn(B-A)$ and let $w$ be a basis vector for $\mfn(C-A)$. By Proposition \ref{prop:nullspaceint}(4), $\{v, v(A+a)\}$ (respectively, $\{w, w(A+a)\}$) is a basis for $\mfN_{AB}$ (respectively, $\mfN_{AC}$). Since $\mfN_{AB}=\mfN_{AC}$, both $w$ and $w(A+a)$ are in $\mfN_{AB}$. Let $c_1, c_2, d_1, d_2 \in F$ be such that $w = c_1 v + c_2 v(A+a)$ and $w(A+a) = d_1 v + d_2 v(A+a)$. Then, $c_1v(A+a) + c_2v(A+a)^2 = d_1v + d_2v(A+a)$, which implies that
\begin{equation}\label{quad relation}
v\big(c_2(A+a)^2+(c_1-d_2)(A+a)-d_1\big) = 0.
\end{equation}
Let $A'= c_2(A+a)^2+(c_1-d_2)(A+a)-d_1$, which is an element of $F[A]$. Since $F[A]$ is a field, either $A'$ is invertible or $A'=0$. From \eqref{quad relation} and the fact that $v \ne 0$, we conclude that $A'=0$. Thus, $c_2=0$, $d_1=0$, and $c_1=d_2$. So, $w=c_1 v$ and hence $\mfn(B-A)=\mfn(C-A)$.

(v) $\Rightarrow$ (i) Let $v$ be nonzero and such that $vA=vB=vC$. Let $w=vA$. Then, 
\begin{equation*}
\begin{bmatrix} -w&v&0\end{bmatrix}\cdot \mcV(A,B,C) = \begin{bmatrix} 0&0&0 \end{bmatrix},
\end{equation*} so $\mcV(A,B,C)$ is not invertible.

(ii) $\Rightarrow$ (iv) We prove the contrapositive. Assume that $\mfn(B-A) \ne \mfn(C-A)$. Then, since these are linear subspaces of $F^3$, $\mfn(B-A) \cap \mfn(C-A) = \{0\}$. Furthermore, $\mfN_{AB} \ne \mfN_{AC}$. But, both $\mfN_{AB}$ and $\mfN_{AC}$ are 2-dimensional subspaces of $F^3$, so $\dim(\mfN_{AB} \cap \mfN_{AC})=1$.

Let $v_1, v_2 \in F^3$ be such that $\begin{bmatrix}v_1 & v_2 \end{bmatrix} \cdot \mcV_2 = 0$, (here, $\begin{bmatrix}v_1 & v_2 \end{bmatrix}$ is the $1 \times 6$ matrix formed from $v_1$ and $v_2$). We wish to show that $v_1=v_2=0$. The fact that $\begin{bmatrix}v_1 & v_2 \end{bmatrix} \cdot \mcV_2 = 0$ means $v_1(B-A) = -v_2(B^2-A^2)$ and $v_1(C-A) = -v_2(C^2-A^2)$. By Lemma \ref{lem:Rspacelem}(3), both $v_1$ and $v_2$ are in $\mfN_{AB} \cap \mfN_{AC}$. Since this subspace is 1-dimensional, $v_1 = \lambda v_2$ for some $\lambda \in F$. This means that $v_2$ lies in both $\mfn(B^2-A^2-\lambda(B-A))$ and $\mfn(C^2-A^2-\lambda(C-A))$. By Lemma \ref{lem:1dimnull}, both of these nullspaces are 1-dimensional, so either they are equal or their intersection is $\{0\}$.

Suppose that $\mfn(B^2-A^2-\lambda(B-A)) = \mfn(C^2-A^2-\lambda(C-A))$. By Proposition \ref{prop:nullspaceint}(1), $\mfn(B-A) \cdot (A+a+\lambda) = \mfn(C-A) \cdot (A+a+\lambda)$, which means that $\mfn(B-A)=\mfn(C-A)$, a contradiction. So, the intersection of $\mfn(B^2-A^2-\lambda(B-A))$ and $\mfn(C^2-A^2-\lambda(C-A))$ must be the zero space. Thus, $v_1=v_2=0$, and $\mcV_2$ is invertible.
\end{proof}

We can use Theorem \ref{thm:singular case} and Lemma \ref{lem:E_A,v} to give a sufficient condition for certain subsets of $\mcC(m)$ to be core. Recall from Definition \ref{def:E_A,v} that for $A \in \mcC(m)$ and a nonzero $v \in F^3$, $E_{A,v}$ is defined to be $E_{A,v} = \{ B \in \mcC(m) : vB=vA\}$.

\begin{Lem}\label{lem:Eav intersection}
Let $A \in \mcC(m)$ and let $v, w \in F^3$ be nonzero.
\begin{enumerate}[(1)]
\item $E_{A,v}=E_{A,w}$ if and only if $v$ and $w$ are scalar multiples.
\item If $E_{A,v} \ne E_{A,w}$, then $E_{A,v} \cap E_{A,w} = \{A\}$.
\end{enumerate}
\end{Lem}
\begin{proof}
For (1), assume that $E_{A,v} = E_{A,w}$ and let $B \in E_{A,v}$ such that $B \ne A$. Then, both $v$ and $w$ are in $\mfn(B-A)$. By Lemma \ref{lem:1dimnull}, $\mfn(B-A)$ is 1-dimensional, so $v$ and $w$ are scalar multiples. The converse is clear.

For (2), assume that $E_{A,v} \cap E_{A,w}$ contains some matrix $B \ne A$. Then, $v, w \in \mfn(B-A)$, and hence are scalar multiples. By part (1), $E_{A,v} = E_{A,w}$.
\end{proof}

\begin{Cor}\label{cor:singular case}
Let $S \subseteq \mcC(m)$. 
\begin{enumerate}[(1)]
\item If there exist distinct $A, B, C \in S$ such that $B-A$ is singular, $C-A$ is singular, and $\mfn(B-A) \ne \mfn(C-A)$, then $S$ is core.

\item If $F = \F_q$, $|S| \ge q^3-q^2+1$, and there exists $A \in S$ such that $B-A$ is singular for all $B \in S$, then $S$ is core.

\end{enumerate}
\end{Cor}
\begin{proof}
For (1), if such $A$, $B$, and $C$ exist, then by Theorem \ref{thm:singular case} the Vandermonde matrix $\mcV(A,B,C)$ is invertible. Hence, $S$ is core by Lemma \ref{lem:Vandy}. For (2), assume that $F=\F_q$, $|S| \ge q^3-q^2+1$, and that $S$ contains the desired matrix $A$. Choose $B \in S \setminus\{A\}$. Then, $\mfn(B-A) \ne \{0\}$, so there exists a nonzero $v \in \F_q^3$ such that $B \in E_{A,v}$. By Lemma \ref{lem:E_A,v}, $|E_{A,v}|=q^3-q^2$, so there exists $C \in S$ such that $C \notin E_{A,v}$. 

Suppose that $w$ is a nonzero vector such that $wA=wB=wC$. Then, $B \in E_{A,v} \cap E_{A,w}$, so by Lemma \ref{lem:Eav intersection}, $v$ and $w$ must be scalar multiples and $E_{A,v} = E_{A,w}$. Hence, $C \in E_{A,v}$, a contradiction. We conclude that such a vector $w$ does not exist. Therefore, by Theorem \ref{thm:singular case}, $\mcV(A,B,C)$ is invertible and $S$ is core by Lemma \ref{lem:Vandy}.
\end{proof}

\section{Mixed Case}\label{mixed section}

Recall the notation and exposition given at the start of Section \ref{singular section}. In this section, we will consider the matrices $\mcV_1$ and $\mcV_2$ under the assumption that one of $B-A$ or $C-A$ is invertible, and the other is not. As we will show, in this situation the Vandermonde matrix $\mcV(A,B,C)$ is always invertible. Hence, any subset $S \subseteq \mcC(m)$ containing such matrices $A$, $B$, and $C$ is core.

\begin{Lem}\label{lem:mixed}
Let $A, B \in \mcC(m)$ such that $B-A$ is invertible. Let $a_1, a_2, a_3 \in F$, and let
\begin{equation*}
M:= a_1(B-A) + a_2(A+a)(B-A) - a_2(B^2-A^2) + a_3(A+a)(B^2-A^2).
\end{equation*}
Then, either $M$ is invertible or $M=0$. Furthermore, $M=0$ if and only if $a_1=a_2=a_3=0$.
\end{Lem}
\begin{proof}
Let $M_1 := a_1(B-A)$, $M_2:= a_2(A+a)(B-A) - a_2(B^2-A^2)$, and $M_3:=a_3(A+a)(B^2-A^2)$. We will first show that each of these three matrices is left divisible by $B-A$. This is clear for $M_1$.

For $M_2$, we have
\begin{align*}
M_2 &= a_2\big((A+a)(B-A)-(B^2-A^2)\big)\\
&= a_2\big(AB-A^2+a(B-A)-B^2+A^2\big)\\
&= a_2(B-A)(-B+a).
\end{align*}
Next, for $M_3$,
\begin{equation*}
M_3 = a_3(A+a)(B^2-A^2) = a_3(AB^2 - A^3 + aB^2 - aA^2).
\end{equation*}
Since $A, B \in \mcC(m)$, we know that $m(A)=m(B)=0$. So, $aB^2 = -B^3-bB-c$ and $-A^3-aA^2 = bA+c$, which yields
\begin{align*}
M_3 &= a_3(AB^2 - A^3 + aB^2 - aA^2)\\
&= a_3(AB^2- B^3-bB-c +bA+c)\\
&= a_3\big((B-A)(-B^2) + (B-A)(-b)\big)\\
&= a_3(B-A)(-B^2-b).
\end{align*}
Let $M_0 = -a_3B^2 - a_2B+a_1+aa_2-ba_3$. Then,
\begin{align*}
M &= M_1+M_2+M_3\\
&= (B-A)(a_1 + a_2(-B+a)+a_3(-B^2-b))\\
&= (B-A)M_0.
\end{align*}
Since $B-A$ is invertible, this shows that $M$ is invertible if and only if $M_0$ is invertible, and $M=0$ if and only if $M_0=0$. Note that $M_0$ is an element of $F[B]$, and $F[B]$ is a field because $m$ is irreducible. So, $M_0$ is either invertible or is 0. Moreover, $M_0 = 0$ if and only if $a_3$, $a_2$, and $a_1+ aa_2-ba_3$ are all 0, which occurs if and only if $a_1=a_2=a_3=0$.
\end{proof}

\begin{Thm}\label{thm:mixed case}
Let $A, B, C \in \mcC(m)$ be distinct. If $B-A$ is invertible and $C-A$ is singular, then $\mcV(A,B,C)$ is invertible.
\end{Thm}
\begin{proof}
Let $\mcV_2 = \begin{bmatrix} B-A & C-A\\ B^2-A^2 & C^2-A^2 \end{bmatrix}$. As described at the beginning of Section \ref{singular section}, $\mcV(A,B,C)$ is invertible if and only if $\mcV_2$ is invertible. Let $v_1, v_2 \in F^3$ be such that $\begin{bmatrix}v_1 & v_2 \end{bmatrix} \mcV_2= 0$. It suffices to show that $v_1=v_2=0$. Since $\begin{bmatrix}v_1 & v_2 \end{bmatrix} \mcV_2= 0$, we have $v_1(C-A) =-v_2(C^2-A^2)$. By Lemma \ref{lem:Rspacelem}(3), both $v_1$ and $v_2$ are in $\mfN_{AC}$.

Let $v \in \mfn(C-A)$ be nonzero. By Proposition \ref{prop:nullspaceint}, $v(A+a) \in \mfn(C^2-A^2)$ and $\{v, v(A+a)\}$ is a basis of $\mfN_{AC}$. So, $v_1 = c_1 v + c_2v(A+a)$ and $v_2= d_1 v + d_2 v(A+a)$ for some $c_1, c_2, d_1, d_2 \in F$. Since $v(C-A)=0$ and $v(A+a)(C^2-A^2)=0$, we get 
\begin{equation}\label{mixed eq 1}
0 = v_1(C-A) +v_2(C^2-A^2) = c_2v(A+a)(C-A)+d_1v(C^2-A^2).
\end{equation}
Now, $C^2-A^2 = (A+a)(C-A) + (C-A)(C-a)$, and $v \in \mfn(C-A)$, so $v(C^2-A^2) = v(A+a)(C-A)$. Thus, \eqref{mixed eq 1} becomes $0 = (c_2+d_1)v(A+a)(C-A)$. Since $v(A+a) \notin \mfn(C-A)$, this means that $d_1=-c_2$.

Let $M$ be the following matrix:
\begin{equation*}
M:= c_1(B-A) + c_2(A+a)(B-A) - c_2(B^2-A^2) + d_2(A+a)(B^2-A^2).
\end{equation*}
Then, the equation $v_1(B-A) + v_2(B^2-A^2)=0$ (which holds because $\begin{bmatrix}v_1&v_2\end{bmatrix}\mcV_2=0$) is equivalent to $vM=0$. By Lemma \ref{lem:mixed}, $M$ is either invertible or is 0. Because $v \ne 0$, $M$ must equal 0. Hence, $c_1=c_2=d_2=0$. It follows that $v_1=v_2=0$. Thus, $\mcV_2$---and therefore $\mcV(A,B,C)$---is invertible.
\end{proof}

\begin{Cor}\label{cor:mixed case core}
Let $S \subseteq \mcC(m)$. 
\begin{enumerate}[(1)]
\item If there exist distinct $A, B, C \in S$ such that $B-A$ is invertible and $C-A$ is singular, then $S$ is core. 
\item If there exist distinct $A, B, C \in S$ such that $\mfn(B-A) \ne \mfn(C-A)$, then $S$ is core.
\item If $F=\F_q$, then there exist core subsets of $\mcC(m)$ of cardinality $n$ for all $3 \le n \le (q^3-q^2)(q^3-q)$.
\end{enumerate}
\end{Cor}
\begin{proof}
(1) This follows from Theorem \ref{thm:mixed case} and Lemma \ref{lem:Vandy}. 

(2) Assume that such matrices $A$, $B$, and $C$ exist. If both $B-A$ and $C-A$ are invertible, then $\mfn(B-A) = \{0\} = \mfn(C-A)$, a contradiction. So, either both $B-A$ and $C-A$ are singular, or exactly one of $B-A$ or $C-A$ is singular and the other is invertible. In the former case, $S$ is core by Corollary \ref{cor:singular case}, and in the latter case $S$ is core by part (1).

(3) Assume that $F=\F_q$. Fix $A \in \mcC(m)$. Take $B=A^q$. Since $B$ is the image of $A$ under the Frobenius automorphism on $\F_q[A]$, $m(B)=0$, and hence $B \in \mcC(m)$. Since $B-A$ is a nonzero element of $\F_q[A]$, it is invertible. Next, take $C$ to be any matrix in one of the sets $E_{A,v} \setminus \{A\}$. Then, $C-A$ is singular. Hence, $\{A, B, C\}$---and any subset of $\mcC(m)$ containing this set---is core.
\end{proof}

\section{Invertible Differences}\label{invertible section}

It remains to consider the case of $\mcV(A,B,C)$ where both $B-A$ and $C-A$ are invertible. Since both $B$ and $C$ are conjugates of $A$, we are thus led to the study of similar matrices whose difference is invertible. This, in turn, has a close connection to invertible matrix commutators, a topic that has attracted interest in the past \cite{GurLan, KhuranaLam, Marq, MartinsSilva1, MartinsSilva2, Uhlig}. For matrices $X, Y \in M_3(F)$, the additive commutator (or Lie bracket) of $X$ and $Y$ is $[X, Y] = XY-YX$. When $U$ is invertible, we have $UAU^{-1} - A = (UA-AU)U^{-1}$. Thus, we obtain the fundamental relationship
\begin{equation}\label{sim and comm}
UAU^{-1} - A \text{ is invertible if and only if } [U,A] \text{ is invertible}.
\end{equation}

\begin{Lem}\label{lem:inv vandy lemma}
Let $A, B, C \in \mcC(m)$ such that both $B-A$ and $C-A$ are invertible. Let $U_1, U_2 \in \GL(3,F)$ be such that $B=U_1 A U_1^{-1}$ and $C=U_2 A U_2^{-1}$. The following are equivalent.
\begin{enumerate}[(i)]
\item $\mcV(A,B,C)$ is invertible.
\item $C^2-A^2 - (B^2-A^2)(B-A)^{-1}(C-A)$ is invertible.
\item $[U_2,A]A[U_2,A]^{-1} - [U_1,A]A[U_1,A]^{-1}$ is invertible.
\end{enumerate}
\end{Lem}
\begin{proof}
Let $\mcV:=\mcV(A,B,C)$ and $X := C^2-A^2-(B^2-A^2)(B-A)^{-1}(C-A)$. As in Section \ref{singular section}, we may perform block row operations on $\mcV$ to obtain $\mcV_2:=\begin{bmatrix}B-A&C-A\\B^2-A^2&C^2-A^2\end{bmatrix}$, and $\mcV$ is invertible if and only if $\mcV_2$ is invertible. Since $B-A$ is invertible, we may further row reduce $\mcV_2$ to produce the matrix
\begin{equation*}
\mcV_3:=\begin{bmatrix}I&(B-A)^{-1}(C-A)\\0&X\end{bmatrix}.
\end{equation*}
Then, $\mcV$ is invertible if and only if $\mcV_3$ is invertible, which holds if and only if $X$ is invertible.

Next, note that $B^2-A^2 = A(B-A) + (B-A)B$. This means that
\begin{equation*}
(B^2-A^2)(B-A)^{-1} = A + (B-A)B(B-A)^{-1},
\end{equation*}
and one may easily check that $(B-A)B(B-A)^{-1} = [U_1,A]A[U_1,A]^{-1}$. Analogous relations hold when $B$ is replaced with $C$. 

Finally, $X$ is invertible if and only if $X(C-A)^{-1}$ is invertible. But, in light of the previous paragraph,
\begin{align*}
X(C-A)^{-1} &= (C^2-A^2)(C-A)^{-1}-(B^2-A^2)(B-A)^{-1}\\
&= [U_2,A]A[U_2,A]^{-1} - [U_1,A]A[U_1,A]^{-1},
\end{align*}
which completes the proof.
\end{proof}

Condition (iii) of Lemma \ref{lem:inv vandy lemma} shows that the invertibility of the Vandermonde matrix $\mcV(A,B,C)$ can be determined by examining differences of certain elements of $\mcC(m)$. 

\begin{Def}\label{def:invcA}
A subset $S \subseteq \GL(3,F)$ is said to have the \textit{invertible difference property} (or is said to be an \textit{IDP set}) if $A - B \in \GL(3,F)$ for all $A, B \in S$ such that $A \ne B$. For $A \in \mcC(m)$, we define the following sets of matrices:
\begin{align*}
\mcD_A &:=\{B \in \mcC(m) : B-A \text{ is invertible}\},\\
\mcU_A &:=\{U \in \GL(3,F) : [U,A] \text{ is invertible}\},\\
\mcB_A &:=\{[U,A] : U \in \mcU_A\}, \text{ and}\\
\mcS_A &:=\{[U,A]A[U,A]^{-1} : [U,A] \in \mcB_A\}.
\end{align*}
\end{Def}

We now proceed to study the sets introduced in Definition \ref{def:invcA}, with the goal of showing that $\mcS_A$ is an IDP set for all $A \in \mcC(m)$. 

\begin{Lem}\label{lem:conjIDP} Let $\mathcal{T} \subseteq \GL(3,F)$, let $U \in \GL(3,F)$, and let $A \in \mcC(m)$. 
\begin{enumerate}[(1)]
\item $\mathcal{T}$ is an IDP set if and only if $U\mathcal{T}$ is an IDP set, if and only if $\mathcal{T}U$ is an IDP set.
\item We have $U(\mcD_A)U^{-1} = \mcD_{UAU^{-1}}$, and likewise for $\mcU_A$, $\mcB_A$, and $\mcS_A$. 
\item For all $A, B \in \mcC(m)$, $\mcS_A$ is an IDP set if and only if $\mcS_B$ is an IDP set.
\end{enumerate}
\end{Lem}
\begin{proof}
Parts (1) and (2) are straightforward to verify, and (3) follows from (1) and (2).
\end{proof}

Over a finite field, we can determine the cardinality of each set in Definition \ref{def:invcA}.

\begin{Prop}\label{prop:counting}
Assume that $F = \F_q$ and let $A \in \mcC(m)$.
\begin{enumerate}[(1)]
\item $|\mcD_A| = (q^3-q^2-q)(q^3-q^2-q-1)$.
\item $|\mcU_A| = (q^3-1)(q^3-q^2-q)(q^3-q^2-q-1)$.
\item $|\mcB_A| = (q^3-1)(q^3-q^2-q)$.
\item $|\mcS_A| = q^3-q^2-q$.
\end{enumerate}
\end{Prop}
\begin{proof} Throughout, let $G = \GL(3,q)$ and let $\mcF = \F_q[A]$, which is a field of order $q^3$.

(1) First, we will count the number of matrices $B \in \mcC(m)$ that are not in $\mcD_A$. Clearly, $B \in \mcC(m) \setminus \mcD_A$ if and only if there exists a nonzero vector $v$ such that $v(B-A)=0$, if and only if $B \in E_{A,v}$. By Lemma \ref{lem:Eav intersection}, there is a distinct set $E_{A,v}$ for each linear subspace of $\F_q^3$, of which there are $q^2+q+1$. Moreover, when $v$ and $w$ are not scalar multiples, we have $E_{A,v} \cap E_{A,w} = \{A\}$. Thus, $\mcC(m) \setminus \mcD_A = \bigcup_{v \ne 0} E_{A,v}$. By Lemma \ref{lem:E_A,v}, $|E_{A,v}\setminus\{A\}|=q^3-q^2-1$, so
\begin{equation*}
|\mcC(m) \setminus \mcD_A| = (q^2+q+1)(q^3-q^2-1)+1.
\end{equation*}
Next, since $|\mcC(m)| = (q^3-q^2)(q^3-q)$, we have
\begin{align*}
|\mcD_A| &= (q^3-q^2)(q^3-q) - \big((q^2+q+1)(q^3-q^2-1)+1\big)\\
&= q^6-2q^5-q^4+q^3+2q^2+q\\
&= (q^3-q^2-q)(q^3-q^2-q-1),
\end{align*}
as desired.

(2) Let $\msC$ be the centralizer of $A$ in $G$. Since $m$ is irreducible, $\msC$ is equal to the unit group of $\mcF$ and has order $q^3-1$. For matrices $U, V \in G$, it is clear that $UAU^{-1}=VAV^{-1}$ if and only if $V\in U\msC$. So, given $B \in \mcC(m)$, the number of $V \in G$ such that $B=VAV^{-1}$ is equal to $|\msC|=q^3-1$.

Now, \eqref{sim and comm} holds for any $U \in G$, so 
\begin{align*}
|\mcU_A| &= \big|\{U \in G : UAU^{-1}-A \in G\}\big|\\
&= \big|\{U \in G : UAU^{-1} \in \mcD_A\}\big|\\
&= \big|\mcD_A\big| \cdot (q^3-1)\\
&= (q^3-q^2-q)(q^3-q^2-q-1)(q^3-1).
\end{align*}

(3) For each $U \in \mcU_A$, let $n_U$ be the number of matrices $V \in \mcU_A$ such that $[U,A]=[V,A]$. If $n_U$ is the same for all $U \in \mcU_A$, then $|\mcB_A| = |\mcU_A|/n_u$. We will show that $n_U = q^3-q^2-q-1$ for every $U \in \mcU_A$.

Given $X,Y \in M_3(\F_q)$, we have $[X,A] = [Y,A]$ if and only if $X-Y \in \msC \cup \{0\} = \mcF$, and this holds if and only if $X = Y + \alpha$ for some $\alpha \in \mcF$. For each $U \in \mcU_A$, define
\begin{equation*}
\mcG_U := \{\alpha \in \mcF : U+\alpha \in G\} \quad \text{and} \quad \wGU := \{\alpha \in \mcF : U+\alpha \notin G\}.
\end{equation*}
Then, $n_U = |\mcG_U| = q^3-|\wGU|$. We will show that $|\wGU|=q^2+q+1$, which is the number of linear subspaces of $\F_q^3$. To see this, let $\mcL$ be the set of all linear subspaces of $\F_q$. Define $L: \wGU \to \mcL$ by $L(\alpha) = \mfn(U+\alpha)$. We will show that $L$ is a bijection.

First, we will check that $L$ is well-defined. Let $\alpha \in \wGU$ and consider the rank of $U+\alpha$. By \cite[Thm.\ 4]{GurLan}, if the rank of $U+\alpha$ is 1, then $[U+\alpha, X]$ is singular for all $X \in M_3(\F_q)$. Since $[U+\alpha,A] = [U,A]$ is invertible, the rank of $U+\alpha$ is not 1. Because $U+\alpha$ is a nonzero $3 \times 3$ singular matrix, the rank of $U+\alpha$ must be 2. Thus, $\dim(\mfn(U+\alpha))=1$, and $L: \wGU \to \mcL$ is well-defined.

Next, for injectivity, let $\alpha_1, \alpha_2 \in \wGU$ be such that $L(\alpha_1)=L(\alpha_2)$. Then, $\mfn(U+\alpha_1)=\mfn(U+\alpha_2)$, so there exists a nonzero vector $v \in \F_q^3$ such that $v(U+\alpha_1)=0=v(U+\alpha_2)$. It follows that $v(\alpha_1-\alpha_2)=0$; but, $\alpha_1 - \alpha_2 \in \mcF$, so this implies that $\alpha_1=\alpha_2$.

Finally, let $W \in \mcL$, and assume $W$ is spanned by $w$. Since $\mcF$ acts transitively on $\F_q^3$, there exists an invertible matrix $\beta \in \mcF$ such that $wU=(-w)\beta$. Hence, $U+\beta \notin G$ and $w \in \mfn(U+\beta)$. As noted above, $\dim(\mfn(U+\beta))=1$, so $L(\beta)=\mfn(U+\beta)=W$. It follows that $|\wGU|=|\mcL|=q^2+q+1$, and we conclude that $n_U = q^3-q^2-q-1$.

(4) As noted in the proof of (2), two conjugates $[U_1,A]A[U_1,A]^{-1}$ and $[U_2,A]A[U_2,A]^{-1}$ are equal if and only if $[U_2,A]=[U_1,A]\alpha$ for some $\alpha \in \msC$. Since $\alpha$ commutes with $A$, we have
\begin{equation*}
[U_1,A]\alpha = U_1 A \alpha - A U_1 \alpha = U_1\alpha A - AU_1\alpha = [U_1 \alpha, A].
\end{equation*}
Thus, for each $[U,A] \in \mcB_A$, there are exactly $|\msC|=q^3-1$ elements of $\mcB_A$ that produce the same matrix in $\mcS_A$. In other words, $|\mcS_A| = |\mcB_A|/(q^3-1)=q^3-q^2-q$.
\end{proof}

From here, we will assume that $m$ is a separable polynomial. Note that since $\deg m=3$, in order for $m$ to be inseparable it is necessary that $F$ is an infinite field of characteristic 3; in all other cases, $m$ will be separable.

\begin{Not}\label{not:separable}
Let $\K$ be the splitting field of $m$ over $F$, and let $\alpha$ be a root of $m$ in $K$. Assuming that $m$ is a separable polynomial, let $\sigma \in \Gal(\K/F)$ be such that $\sigma$ cyclically permutes the roots of $m$. Then, the three roots of $m$ are $\alpha$, $\sigma(\alpha)$, and $\sigma^2(\alpha)$. (We remark that when $F = \F_q$, we may choose $\sigma$ to be the Frobenius automorphism, so that $m$ has roots $\alpha$, $\alpha^q$, and $\alpha^{q^2}$). Define
\begin{equation*}
P:= \begin{bmatrix} 1 & \alpha & \alpha^2 \\ 1 & \sigma(\alpha) & \sigma(\alpha^2) \\ 1 & \sigma^2(\alpha) & \sigma^2(\alpha^2) \end{bmatrix} \in \GL(3,\K).
\end{equation*}
Finally, given $d_1, d_2, d_3 \in \K$, let $\Diag(d_1, d_2, d_3)$ be the $3 \times 3$ matrix with diagonal entries $d_1$, $d_2$, and $d_3$, and all nondiagonal entries equal to 0. 
\end{Not}

We point out the interesting fact that the set $\mcU_A$ remains the same as $A$ runs through the nonscalar matrices that commute with $A$.

\begin{Lem}
\label{lem:UA=UB}
Assume that $m$ is separable. Let $A \in \mcC(m)$ and let $B \in F[A]$, where $B$ is not a scalar matrix.  Then, $\mcU_B = \mcU_A$.  
\end{Lem}

\begin{proof}
Since $m$ is separable, it has distinct roots and splits in $\K$. Consequently, there exists $Q \in \GL(3, \K)$ such that $D := QAQ^{-1}$ is a diagonal matrix, and no two diagonal entries of $D$ are equal.  Moreover, since $B \in F[A]$, $A$ and $B$ commute, which means $A$ and $B$ are simultaneously diagonalizable, i.e., we may assume $D' := QBQ^{-1}$ is also a diagonal matrix with distinct diagonal entries.
 
Let $U \in \GL(3,F)$.  Then, $[U,A]$ is invertible if and only if $[X,D]$ is invertible, where $X := QUQ^{-1}$ (and, similarly, $[U,B]$ is invertible if and only if $[X,D']$ is invertible).  If $X$ has entries $x_{ij}$ for $1 \le i, j \le 3$ and $D:= \Diag(d_1, d_2, d_3)$, then
\begin{equation}\label{eq:diagcomm}
[X,D] = \begin{bmatrix}
             0                 & (d_2 - d_1)x_{12} & (d_3 - d_1)x_{13} \\
             (d_1 - d_2)x_{21} & 0                 & (d_3 - d_2)x_{23} \\
             (d_1 - d_3)x_{31} & (d_2 - d_3)x_{32} & 0 \\
            \end{bmatrix}.
\end{equation}
In particular, $\det [X,D] = (d_2 - d_1)(d_3 - d_2)(d_1 - d_3)(x_{12}x_{23}x_{31} - x_{13}x_{21}x_{32})$, and so $[X,D]$ is invertible if and only if $x_{12}x_{23}x_{31} - x_{13}x_{21}x_{32} \neq 0$.  A similar calculation shows that $[X,D']$ is invertible if and only if $x_{12}x_{23}x_{31} - x_{13}x_{21}x_{32} \neq 0$, and, therefore, $[U,A]$ is invertible if and only if $[U,B]$ is invertible.  The result follows.
\end{proof}

Returning to the question of whether $\mcS_A$ is an IDP set, by Lemma \ref{lem:conjIDP}(3) it suffices to consider the case where $A$ is the companion matrix of $m$, that is, where 
\[A=\left[\begin{matrix}0&0&-c\\1&0&-b\\0&1&-a\end{matrix}\right].\] Following Notation \ref{not:separable}, assume that $m$ is separable and has roots $\alpha$, $\sigma(\alpha)$, and $\sigma^2(\alpha)$. Then, the rows of the matrix $P$ are eigenvectors of $A$, and $PAP^{-1} = \Diag(\alpha, \sigma(\alpha), \sigma^2(\alpha))$. Note that since $P$ is a Vandermonde matrix, its determinant is $\det P = (\sigma(\alpha)-\alpha)(\sigma^2(\alpha)-\alpha)(\sigma^2(\alpha)-\sigma(\alpha))$. Since this product is stable under $\sigma$, $\det P \in F$.

\begin{Lem}\label{lem:Xentries}
Assume that $m$ is separable. Let $U \in M_3(F)$, let $X=PUP^{-1}$, and assume that $X$ has entries $x_{ij}$ for $1 \le i, j \le 3$. Then, $x_{22}=\sigma(x_{11})$, $x_{33} = \sigma^2(x_{11})$, $x_{23}=\sigma(x_{12})$, $x_{31}=\sigma^2(x_{12})$, $x_{21}=\sigma(x_{13})$, and $x_{32}=\sigma^2(x_{13})$.
\end{Lem}
\begin{proof}
We summarize the results necessary to verify that $x_{23}=\sigma(x_{12})$ and $x_{31}=\sigma^2(x_{12})$. The other cases are similar and are left to the reader. For readability, let $\beta = \sigma(\alpha)$ and $\gamma=\sigma(\beta)=\sigma^2(\alpha)$. Note that each element of $F$ is fixed by $\sigma$, and $\sigma^3$ is the identity function on $\K$.

Now, as noted above, $\det P \in F$, so it suffices to show that the desired relations hold for $(\det P)X$. Equivalently, we can consider $PU(\text{adj}(P))$, where $\text{adj}(P)$ is the adjugate matrix of $P$. One may compute that
\begin{equation*}
\text{adj}(P) = 
\begin{bmatrix}
\beta\gamma(\gamma-\beta) & \alpha\gamma(\alpha-\gamma) & \alpha\beta(\beta-\alpha)\\
\beta^2-\gamma^2 & \gamma^2-\alpha^2 & \alpha^2-\beta^2\\
\gamma-\beta & \alpha-\gamma & \beta-\alpha
\end{bmatrix}.
\end{equation*}

Next, let $U$ have entries $u_{ij}$ for $1 \le i, j \le k$. Then, the $(1,2)$-entry of $(\det P)X$ is
\begin{align*}
x_{12} =  \alpha \gamma (\alpha&-\gamma) u_{11} + (\gamma^2-\alpha^2) u_{12} + (\alpha-\gamma) u_{13}\\ 
&+ \alpha^2 \gamma (\alpha-\gamma) u_{21} + \alpha (\gamma^2-\alpha^2) u_{22} + \alpha (\alpha-\gamma) u_{23}\\ 
&+ \alpha^3 \gamma (\alpha-\gamma) u_{31} + \alpha^2 (\gamma^2-\alpha^2) u_{32} + \alpha^2 (\alpha-\gamma) u_{33},
\end{align*}
while the $(2,3)$-entry is 
\begin{align*}
x_{23} = \alpha \beta (\beta&-\alpha) u_{11} + (\alpha^2-\beta^2) u_{12} + (\beta-\alpha) u_{13}\\ 
&+ \alpha \beta^2 (\beta-\alpha) u_{21} + \beta (\alpha^2-\beta^2) u_{22} + \beta (\beta-\alpha) u_{23}\\ 
&+ \alpha \beta^3 (\beta-\alpha) u_{31} + \beta^2 (\alpha^2-\beta^2) u_{32} + \beta^2 (\beta-\alpha) u_{33},
\end{align*}
and the $(3,1)$-entry is
\begin{align*}
x_{31} = \beta \gamma (\gamma&-\beta) u_{11} + (\beta^2-\gamma^2) u_{12} + (\gamma-\beta) u_{13}\\ 
&+ \beta \gamma^2 (\gamma-\beta) u_{21} + \gamma (\beta^2-\gamma^2) u_{22} + \gamma (\gamma-\beta) u_{23}\\ 
&+ \beta \gamma^3 (\gamma-\beta) u_{31} + \gamma^2 (\beta^2-\gamma^2) u_{32} + \gamma^2 (\gamma-\beta) u_{33}.
\end{align*}
Since $\sigma$ cyclically permutes $\alpha$, $\beta$, and $\gamma$, it is now clear that $x_{23}=\sigma(x_{12})$ and $x_{31}=\sigma^2(x_{12})$. The remaining relations can be checked by performing similar calculations.
\end{proof}

\begin{Lem}\label{lem:canonical form}
Assume that $m$ is separable. Let $A$ be the companion matrix of $m$, let $D=PAP^{-1} \in \GL(3,\K)$, and let $Q \in \mcS_A$. Then, there exists a matrix $X \in M_3(\K)$ such that $[X,D]$ is invertible, $PQP^{-1} = [X,D]D[X,D]^{-1}$, and $X$ has one of the following forms:
\begin{enumerate}[(i)]
\item $X = \begin{bmatrix} 0 & 1 & 0 \\ 0 & 0 & 1\\ 1 & 0 & 0 \end{bmatrix}$, 
\item $X = \begin{bmatrix} 0 & 0 & 1 \\ 1 & 0 & 0\\ 0 & 1 & 0 \end{bmatrix}$, or
\item $X = \begin{bmatrix} 0 & 1 & 1 \\ 1 & 0 & x\\ \sigma(x) & \sigma^2(x^{-1}) & 0 \end{bmatrix}$ where $x \in \K$ is nonzero.
\end{enumerate}
\end{Lem}
\begin{proof}
Since $Q \in \mcS_A$, we know that $Q=[U,A]A[U,A]^{-1}$ for some $U \in \GL(3,F)$ such that $[U,A]$ is invertible. Let $X_1 = PUP^{-1} \in \GL(3,\K)$. Then, $[X_1,D] = P[U,A]P^{-1}$, so $[X_1,D]$ is invertible, and
\begin{equation*}
PQP^{-1} = P([U,A]A[U,A]^{-1})P^{-1} = [X_1,D]D[X_1,D]^{-1}.
\end{equation*}
Moreover, by Lemma \ref{lem:Xentries}, we know that for some $x_{11}, x_{12}, x_{13} \in \K$,
\begin{equation*}
X_1 = \begin{bmatrix} x_{11} & x_{12} & x_{13} \\ \sigma(x_{13}) & \sigma(x_{11}) & \sigma(x_{12}) \\ \sigma^2(x_{12}) & \sigma^2(x_{13}) & \sigma^2(x_{11}) \end{bmatrix}.
\end{equation*}
Note that we cannot have $x_{12}=x_{13}=0$, because this would contradict the fact that $[X_1,D]$ is invertible (cf.\ equation \eqref{eq:diagcomm}).

Consider a diagonal matrix $Z \in M_3(\K)$ and an invertible diagonal matrix $R \in \GL(3, \K)$. Since $Z$ and $R$ commute with $D$, we have
\begin{equation*}
[X_1 - Z, D] = [X_1, D] - [Z, D] = [X_1, D]
\end{equation*}
and 
\begin{equation*}
[X_1 R, D] = (X_1 R)D - D(X_1 R) = (X_1 D - D X_1)R = [X_1,D]R.
\end{equation*}
The latter equation implies that 
\begin{equation*}
[X_1 R,D]D[X_1 R,D]^{-1} = [X_1 ,D]RDR^{-1}[X_1 ,D]^{-1} = [X_1 ,D]D[X_1 ,D]^{-1}.
\end{equation*}
Thus, the desired matrix $X$ is $X=(X_1-Z)R$, where $Z = \Diag(x_{11},\sigma(x_{11}),\sigma^2(x_{11}))$ and $R$ is chosen to appropriately scale the columns of $X_1-Z$. Specifically, if $x_{13}=0$, then take $R = \Diag(\sigma^2(x_{12}^{-1}), x_{12}^{-1}, \sigma(x_{12}^{-1}))$; and if $x_{12}=0$, then take $R=\Diag(\sigma(x_{13}^{-1}), \sigma^2(x_{13}^{-1}), x_{13}^{-1})$. These produce forms (i) and (ii), respectively. If both $x_{12}$ and $x_{13}$ are nonzero, then let $R = \Diag(\sigma(x_{13}^{-1}), x_{12}^{-1}, x_{13}^{-1})$. In this case, $X$ will be in form (iii) with $x = \sigma(x_{12})/x_{13}$.
\end{proof}

\begin{Thm}\label{thm:SAIDP}
Assume that $m$ is separable. Let $A \in \mcC(m)$. Then, $\mcS_A$ is an IDP set.
\end{Thm}

\begin{proof}
By Lemma \ref{lem:conjIDP}(3), we may assume without loss of generality that $A$ is the companion matrix of $m$. Let $\K$, $\alpha$, $\sigma$, and $P$ be as in Notation \ref{not:separable}. Then, $D:= PAP^{-1} = \Diag(\alpha, \sigma(\alpha), \sigma^2(\alpha))$. Let $\delta = (\sigma(\alpha)-\alpha)(\sigma^2(\alpha)-\sigma(\alpha))(\alpha-\sigma^2(\alpha))$, which is nonzero.

Let $Q_1, Q_2 \in \mcS_A$ such that $Q_1 \ne Q_2$. Then, there exist matrices $X, Y \in M_3(\K)$ such that $PQ_1P^{-1} = [X,D]D[X,D]^{-1}$, $PQ_2P^{-1} = [Y,D]D[Y,D]^{-1}$, and $X$ and $Y$ each have one of the forms specified in Lemma \ref{lem:canonical form}. Assume that $X$ has entries $x_{ij}$ $(1 \le i, j \le 3)$ and $Y$ has entries $y_{ij}$ $(1 \le i,j \le 3)$. Then, as noted after equation \eqref{eq:diagcomm},
\begin{align*}
\det[X,D] &= \delta(x_{12}x_{23}x_{31} - x_{13}x_{21}x_{32}), \text{ and }\\
\det[Y,D] &= \delta(y_{12}y_{23}y_{31} - y_{13}y_{21}y_{32}).
\end{align*}
Since both $[X,D]$ and $[Y,D]$ are invertible, neither $x_{12}x_{23}x_{31} - x_{13}x_{21}x_{32}$ nor $y_{12}y_{23}y_{31} - y_{13}y_{21}y_{32}$ is zero.

Now, the determinant of $Q_1-Q_2$ is equal to that of $[X,D]D[X,D]^{-1} - [Y,D]D[Y,D]^{-1}$. One may calculate that this latter determinant is 
\begin{equation*}
\frac{\delta(x_{23}y_{13} - x_{13}y_{23})(x_{31}y_{21}-x_{21}y_{31})(x_{32}y_{12} - x_{12}y_{32})}{(x_{12}x_{23}x_{31} - x_{13}x_{21}x_{32})(y_{12}y_{23}y_{31} - y_{13}y_{21}y_{32})}.
\end{equation*}
Assume first that $X$ has form (i). Then, $Q_1-Q_2$ is singular if and only if $y_{13}y_{21}y_{32}=0$. But, this holds if and only if $Y$ has form (i). Thus, $X=Y$ and $Q_1=Q_2$, a contradiction. Hence, $Q_1-Q_2$ is invertible. We obtain the same result whenever either of $X$ or $Y$ has form (i) or (ii). So, for the remainder of the proof we will assume that both $X$ and $Y$ have form (iii). Let $x:=x_{23}$ and $y:=y_{23}$. Then, the determinant of $Q_1-Q_2$ becomes
\begin{equation*}
\frac{\delta(x-y)(\sigma(x)-\sigma(y))(\sigma^2(x^{-1})-\sigma^2(y^{-1}))}{(x\sigma(x)-\sigma^2(x^{-1}))(y\sigma(y)-\sigma^2(y^{-1}))}.
\end{equation*}
From this, it is clear that $Q_1-Q_2$ is singular if and only if one of $x-y$, $\sigma(x)-\sigma(y)$, or $\sigma^2(x^{-1})-\sigma^2(y^{-1})$ is zero. But, each of these conditions implies that $x=y$, and hence that $Q_1=Q_2$. We conclude that $Q_1-Q_2$ is invertible, and therefore $\mcS_A$ is an IDP set.
\end{proof}

\begin{Cor}\label{cor:vandy invertible case}
Assume that $m$ is separable. Let $A, B, C \in \mcC(m)$. If $B-A$ is invertible, $C-A$ is invertible, and $(B-A)B(B-A)^{-1} \ne (C-A)C(C-A)^{-1}$, then $\mcV(A,B,C)$ is invertible.
\end{Cor}
\begin{proof}
Let $Q_B:=(B-A)B(B-A)^{-1}$ and $Q_C:=(C-A)C(C-A)^{-1}$. Then, $Q_B=[U,A]A[U,A]^{-1}$ 
for any $U \in GL(3,F)$ such that $B=UAU^{-1}$, and likewise for $Q_C$. Thus, $Q_B, Q_C \in \mcS_A$. Since $Q_B \ne Q_C$, the difference $Q_B-Q_C$ is invertible by Theorem \ref{thm:SAIDP}. Hence, $\mcV(A,B,C)$ is invertible by Lemma \ref{lem:inv vandy lemma}.
\end{proof}

\begin{Cor}\label{cor:DA core bound}
Let $S \subseteq \mcC(m)$. 
\begin{enumerate}[(1)]
\item Assume that $m$ is separable. If $S$ contains matrices $A$, $B$, and $C$ such that $B-A$ is invertible, $C-A$ is invertible, and $(B-A)B(B-A)^{-1} \ne (C-A)C(C-A)^{-1}$, then $S$ is core.
\item If $F = \F_q$, $|S| \ge q^3-q^2-q+1$, and there exists $A \in S$ such that $S \subseteq \mcD_A \cup \{A\}$, then $S$ is core.
\end{enumerate}
\end{Cor}
\begin{proof}
Part (1) follows from Corollary \ref{cor:vandy invertible case} and Lemma \ref{lem:Vandy}. For (2), assume that $F=\F_q$, $|S| \ge q^3-q^2-q+1$ and that the desired matrix $A$ exists. For each $B \in \mcD_A$, let $Q_B:=[U,A]A[U,A]^{-1}$, where $U \in \GL(3,q)$ is such that $B=UAU^{-1}$. Note that $Q_B$ does not depend on the choice of $U$, since $[U,A]A[U,A]^{-1}=(B-A)B(B-A)^{-1}$. From the counts in Proposition \ref{prop:counting}, we see that for each $Q \in \mcS_A$, there exist $q^3-q^2-q-1$ matrices $B \in \mcD_A$ such that $Q_B=Q$. Now, since $|S\setminus \{A\}| \ge q^3-q^2-q$, we can find distinct matrices $B$ and $C$ in $S \setminus \{A\}$ such that $Q_B \ne Q_C$. Thus, $S$ is core by part (1).
\end{proof}

It is an open question whether the results of this section (in particular Theorem \ref{thm:SAIDP}) hold without the assumption that $m$ is separable. Regardless, we now have everything necessary to prove Theorem \ref{thm:main}. Since that theorem assumes that $F$ is finite, the separability of $m$ is not a concern.

\begin{myproof}[of Theorem \ref{thm:main}]
Assume that $|S| \ge q^3-q^2+1$, and fix $A \in S$. If $S \subseteq \mcD_A \cup \{A\}$, then $S$ is core by Corollary \ref{cor:DA core bound}. So, assume that $S \not\subseteq \mcD_A \cup \{A\}$. Then, there exists $C \in S$ such that $C-A$ is singular. If $B-A$ is singular for all $B \in S$, then $S$ is core by Corollary \ref{cor:singular case}. Otherwise, there exists $B \in S$ such that $B-A$ is invertible, and $S$ is core by Corollary \ref{cor:mixed case core}.
\end{myproof}

\section{Invertible difference graphs}\label{graph section}

We close the paper by presenting some connections between invertible difference property (IDP) sets (see Definition \ref{def:invcA}) and algebraically defined graphs; indeed, many of the results from Section \ref{invertible section} can be restated in terms of graphs.

Maintain the notation from previous sections, and assume that $F = \F_q$ and $m \in \F_q[x]$ is an irreducible cubic.  Define a graph $\Gamma_m$ whose vertex set $V(\Gamma_m)$ is $\mcC(m)$ with adjacency defined by $A \sim B$ if and only if $A - B$ is invertible.  Such algebraically defined graphs have been of great interest to researchers in recent years; see, for example, \cite{HHLS, SM}. When translated into graph theoretical statements, our results on IDP sets illustrate some notable properties of $\Gamma_m$.

\begin{Prop}
\label{prop:graphrestatement}
 Let the graph $\Gamma_m$ be defined over $\F_q$ as above.
\begin{enumerate}[(1)]
\item $|V(\Gamma_m)| = (q^3 - q)(q^3 - q^2)$.
\item $\Gamma_m$ is vertex-transitive.
\item $\Gamma_m$ is regular of degree $(q^3 - q^2 - q)(q^3 - q^2 - q - 1)$.
\item $\Gamma_m$ has a clique of size $q^3 - q^2 - q$.
\end{enumerate}
\end{Prop}

\begin{proof}
Part (1) is a restatement of Equation \eqref{eq:Cmsize}, and (2) is true because the action of $\GL(3,q)$ by conjugation on $\mcC(m)$ preserves adjacency. For (3), each vertex $A \in \mcC(m)$ has degree $|\mcD_A|$ (see Definition \ref{def:invcA}), and $|\mcD_A| = (q^3 - q^2 - q)(q^3 - q^2 - q - 1)$ by Proposition \ref{prop:counting}(1). Finally, by Theorem \ref{thm:SAIDP} the vertices in each set $\mcS_A$ form a clique, and $|\mcS_A| = q^3 - q^2 - q$ by Proposition \ref{prop:counting}(4).
\end{proof}

We point out that for any graph $\Gamma$ whose vertex set is a subset of $\GL(3,q)$ with adjacency defined via invertible differences, the size of a clique in $\Gamma$ cannot exceed $q^3 - 1$. This is because any set of $q^3$ invertible matrices in $\GL(3,q)$ must have two with the same first row. Thus, the clique size achieved in $\Gamma_m$ is near the theoretical absolute upper bound.

Proposition \ref{prop:graphrestatement} is arguably more interesting if one considers the complement graph $\Gamma_m^c$ (that is, edges and non-edges are switched): this graph still has $(q^3 - q)(q^3 - q^2)$ vertices, but $\Gamma_m^c$ is regular of degree $q^5 - 2q^2 - q - 1$ with an independent set of size $q^3 - q^2 - q$ (corresponding to the clique of $\Gamma_m$ from Proposition \ref{prop:graphrestatement}(3)), which is likely of interest to those studying MRD codes and related objects (see, for example, \cite{Sheekey}).   It would be quite interesting both to determine further combinatorial properties of these graphs (e.g., chromatic number, Hamiltonicity, etc.), and to see whether similar phenomena occur with $n \times n$ matrices when $n > 3$.


\end{document}